\documentclass[12pt]{amsart}
\usepackage{amssymb,mathrsfs,amscd,stmaryrd,fullpage}

\newtheorem{thm}{Theorem}
\newtheorem{lem}[thm]{Lemma}
\newtheorem{cor}[thm]{Corollary}

\theoremstyle{definition}

\theoremstyle{remark}

\numberwithin{equation}{section}

\newcommand{\GL}{{\operatorname{GL}}}

\def\HJ#1{\par\medskip\noindent{\bf#1.}\bgroup\it \ }
\def\EHJ{\egroup}

\begin{document}

\title{A note on larger orbit size}

\author{Ping Jin and Yong Yang}

\address{School of Mathematical Sciences, Shangxi University, Taiyuan, 030006, China.}
\email{jinping@sxu.edu.cn}

\address{Department of Mathematics, Texas State University, 601 University Drive, San Marcos, TX 78666, USA and Key Laboratory of Group and Graph Theories and Applications, Chongqing University of Arts and Sciences, Chongqing, China.}
\email{yang@txstate.edu}

\date{}

\begin{abstract}
In this note, we present an improvement on the large orbit result of Halasi and Podoski,
and then answer an open question raised by  Chen, Cossey, Lewis, and Tong-Viet in ~\cite{CCLT}.
\end{abstract}

\maketitle

In ~\cite[Theorem 3.6]{CCLT}, Chen, Cossey, Lewis, and Tong-Viet
considered the degrees of all the irreducible $\pi$-partial characters of $\pi$-separable groups and proved the following.

\begin{thm}\label{pi}
Let $\pi$ be a set of primes, and let $G$ be a finite $\pi$-separable group. If $\varphi(1)^2$ divides $|G : \ker \varphi|$ for all $\varphi \in I_{\pi} (G)$, then one of the following occurs:

{\rm(1)} $G$ has a normal Hall $\pi$-complement $N$ and $G/N$ is nilpotent.

{\rm(2)} There is a prime $p \in \pi$ and a $\pi'$-group $K$ with a faithful module $V$ of characteristic $p$ so that $|C_K(v)| \geq \sqrt{|K|} $ for all $v \in V$ and $G$ has a section isomorphic to $VK$.
\end{thm}

At the end of that paper, the authors also conjectured that case (2) cannot happen. In this short note, we confirm that this is the case.

For our work, we need Theorem 1.1 and Corollary 1.4 of ~\cite{HalasiPodoski},
which we restate here for convenience.

\begin{thm}\label{H-P}
Let $V$ be a finite vector space and $G \leq \GL(V)$ be a coprime linear group. Then there exist $v$, $w \in V$ such that $C_G(v) \cap C_G(w) = 1$.
\end{thm}

\begin{cor}\label{D-N}
If $G$ is a finite group acting faithfully on a finite group $K$ and $(|G|,|K|)=1$,
then there exists $x\in K$ such that $|C_G(x)|\le \sqrt{|G|}$.
\end{cor}

Actually, it suffices to prove that in the situation of the above corollary,
the upper bound $\sqrt{|G|}$ cannot be attained unless $G$ is trivial.
To do this, we need a variation of ~\cite[Theorem 3.34]{Isaacsfinitegroup},
which was first stated in ~\cite[Lemma 2.1]{DOLFINavarro}
by simply referring to the proof of that theorem.
In fact, the argument was embedded in the proof of ~\cite[Theorem 3.34]{Isaacsfinitegroup}), when Isaacs tried to reprove a result about coprime orbit sizes (the main result of ~\cite{Yuster}). Since the argument is nice and important, we include it here.

\begin{lem}\label{isaacs}
Let $G$ be a finite group acting via automorphisms on a finite abelian group $V$ (written additively), and suppose that $(|G|,|V|)=1$.
If $G = C_G(v)C_G(w)$ for some elements $v,w\in V$,
then $C_G(v + w) = C_G(v)\cap C_G(w)$.
\end{lem}
\begin{proof}
Write $u=v+w$, and obverse that $C_G(v)\cap C_G(w)\subseteq C_G(u)$,
so we need to prove the reverse containment.
By symmetry between $v$ and $w$, it suffices to show that $C_G(u)\subseteq C_G(v)$.

Let $W=\langle w^G\rangle$ be the subgroup of $V$ generated by the $G$-orbit
$w^G=\{wg|g\in G\}$. Then $W$ is $G$-invariant. Note that $u\in W+v$ and thus $W+v=W+u$,
which implies that $C_G(u)$ stabilizes the coset $W+v$.
Let $H$ be the set-wise stabilizer in $G$ of $W+v$, so that $C_G(u)\subseteq H$.
To complete the proof, therefore, it suffices to show that $H\subseteq C_G(v)$.
Of course, we have $C_G(v)\subseteq H$.

Now consider the permutation actions on the set $W+v$ of $H$ and $W$. In this situation, Glauberman's lemma (see Lemma 3.24 of \cite{Isaacsfinitegroup}) applies, and we conclude that $W+v$ contains a $H$-invariant element $s$.
So $v\in W+s$, and we can write $v=t+s$ for some element $t\in W$,
which implies that $H\subseteq C_G(v)$ precisely when $H\subseteq C_G(t)$.
In what follows, therefore, we need only prove that $H$ leaves $t$ invariant.

We can define a map $\varphi:W\to V$ by setting
$$\varphi(x)=\sum_{g\in C_G(v)}xg,$$
and observe that $\varphi$ is a group homomorphism from $W$ into $V$.
By hypothesis, $G = C_G(w)C_G(v)$, so $C_G(v)$ acts transitively on the set $w^G$.
If $x\in w^G$, then as $g$ runs over the elements of $C_G(v)$,
each element of $w^G$ occurs equally often in the form $xg$,
and we can deduce that $\varphi(x)$ is a multiple of the sum of the elements of the orbit $w^G$.
This proves that $\varphi(x)$ is $G$-invariant for all $x\in w^G$, and hence $\varphi(W)\subseteq C_V(G)$. In particular, we have $\varphi(t)\in \varphi(W)$ is $G$-invariant.
However, since $C_G(v)\subseteq H$ fixes $v$ and $s$, it follows that $t=v-s$ is $C_G(v)$-invariant,
and thus $\varphi(t)=|C_G(v)|t$.
Note that $|C_G(v)|$ is coprime to $|V|$, and we conclude that $t$ is also a multiple of $\varphi(t)$. It follows that $t$ is $H$-invariant, as desired.
\end{proof}

As promised, we can now improve Corollary \ref{D-N}.
\begin{cor}\label{new}
If $G$ is a nontrivial finite group acting faithfully on a finite group $K$ and $(|G|,|K|)=1$,
then there exists $x\in K$ such that $|C_G(x)|<\sqrt{|G|}$.
\end{cor}
\begin{proof}
As usual, we may assume without loss that $K$ is an abelian group (written additively) by a lemma of B. Hartley and A. Turull (see Theorem 3.31 of \cite{Isaacsfinitegroup}, for example).
Then by Theorem ~\ref{H-P}, there exist $v$, $w \in V$ such that $C_G(v) \cap C_G(w) = 1$,
and we may assume $|C_G(v)| \leq |C_G(w)|$.
Thus we have
$$|C_G(v)|^2 \leq |C_G(v)| |C_G(w)| = |C_G(v) C_G(w)||C_G(v) \cap  C_G(w)| = |C_G(v) C_G(w)| \leq |G|.$$
If $|C_G(v) C_G(w)| < |G|$, then $|C_G(v)|^2 < |G|$ and the result is clear. If $C_G(v) C_G(w) = G$, then  $C_G(v + w) = C_G(v) \cap C_G(w) = 1$ by Lemma ~\ref{isaacs},
and the result is again clear.
\end{proof}

It is easy to see that by Corollary \ref{new},
case (2) in Theorem \ref{pi} indeed cannot occur, and this confirms the conjecture. We now have the following.

\begin{thm}\label{pinew}
Let $\pi$ be a set of primes, and let $G$ be a finite $\pi$-separable group. If $\varphi(1)^2$ divides $|G : \ker \varphi|$ for all $\varphi \in I_{\pi} (G)$, then $G$ has a normal Hall $\pi$-complement $N$ and $G/N$ is nilpotent.
\end{thm}

\section*{Acknowledgement}\label{sec:Acknowledgement}
This work was partially supported by NSFC (Grant Nos. 11671063 and 11671238), the Natural Science Foundation of CSTC (cstc2018jcyjAX0060), and a grant from the Simons Foundation (No 499532, to Y. Yang). 

\end{document}